\documentclass[onefignum,onetabnum]{siamart171218}

\usepackage{lipsum}
\usepackage{amsfonts}
\usepackage{graphicx}
\usepackage{epstopdf}
\usepackage{algorithmic}

\ifpdf
  \DeclareGraphicsExtensions{.eps,.pdf,.png,.jpg}
\else
  \DeclareGraphicsExtensions{.eps}
\fi

\newcommand{\ri}{\mathrm{i}}
\newcommand{\rme}{\mathrm{e}}
\newcommand{\rd}{\, \mathrm{d}}
\newcommand{\CC}{\mathbb{C}}

\newcommand{\NN}{\mathbb{N}}
\newcommand{\RR}{\mathbb{R}}

\newcommand{\Dcal}{\mathcal{D}}
\newcommand{\Ecal}{\mathcal{E}}

\newcommand{\rIm}{\mathrm{Im}}

\newcommand{\wor}{\mathrm{wor}}
\newcommand{\low}{\mathrm{low}}
\newcommand{\up}{\mathrm{up}}

\newcommand{\SE}{\mathrm{SE}}
\newcommand{\DE}{\mathrm{DE}}

\allowdisplaybreaks

\usepackage{xcolor}
\definecolor{darkblue}{RGB}{0,60,180} 
\definecolor{darkgreen}{RGB}{0,130,70}
\definecolor{darkorange}{RGB}{180,60,0}

\newsiamremark{remark}{Remark}
\newsiamremark{hypothesis}{Hypothesis}
\crefname{hypothesis}{Hypothesis}{Hypotheses}
\newsiamthm{claim}{Claim}

\headers{How sharp are error bounds?}{T. Goda, Y. Kazashi, and K. Tanaka}

\title{How sharp are error bounds?\\---lower bounds on quadrature worst-case errors for analytic functions---\thanks{All authors contributed equally to this work. Submitted to the editors DATE.
\funding{The work of T.~G.\ is supported by JSPS KAKENHI Grant Number 23K03210.}}
}
\author{Takashi Goda\thanks{Graduate School of Engineering, University of Tokyo, 7-3-1 Hongo, Bunkyo-ku, Tokyo 113-8656, Japan
(\email{goda@frcer.t.u-tokyo.ac.jp}).}
\and Yoshihito Kazashi\thanks{Department of Mathematics and Statistics, University of Strathclyde, 26 Richmond St., Glasgow G1 1XH, UK
(\email{y.kazashi@strath.ac.uk}).}
\and Ken'ichiro Tanaka\thanks{Department of Mathematical and Computing Science, School of Computing, Institute of Science Tokyo, 2-12-1 Ookayama, Meguro-ku, Tokyo 152-8550, Japan
(\email{kenichiro@c.titech.ac.jp}).}
}

\usepackage{amsopn}

\makeatletter
\newcommand*{\addFileDependency}[1]{
  \typeout{(#1)}
  \@addtofilelist{#1}
  \IfFileExists{#1}{}{\typeout{No file #1.}}
}
\makeatother

\ifpdf
\hypersetup{
  pdftitle={How sharp are error bounds?},
  pdfauthor={T. Goda, Y. Kazashi, and K. Tanaka}
}
\fi

\begin{document}

\maketitle
{\centering\small\textit{Dedicated to the memory of Professor Masaaki Sugihara}\par}

\begin{abstract}
  Numerical integration over the real line for analytic functions is studied. Our main focus is on the sharpness of the error bounds. We first derive two general lower estimates for the worst-case integration error, and then apply these to establish lower bounds for various quadrature rules. These bounds turn out to be either novel or improve upon existing results, leading to lower bounds that closely match upper bounds for various formulas. Specifically, for the suitably truncated trapezoidal rule, we improve upon general lower bounds on the worst-case error obtained by Sugihara [\textit{Numer. Math.}, 75 (1997), pp.~379--395] and provide exceptionally sharp lower bounds apart from a polynomial factor, and in particular show that the worst-case error for the trapezoidal rule by Sugihara is not improvable by more than a polynomial factor. 
  Additionally, our research reveals a discrepancy between the error decay of the trapezoidal rule and Sugihara's lower bound for general numerical integration rules, introducing a new open problem. Moreover, Gauss--Hermite quadrature is proven sub-optimal under the decay conditions on integrands we consider, a result not deducible from upper-bound arguments alone. Furthermore, to establish the near-optimality of the suitably scaled Gauss--Legendre and Clenshaw--Curtis quadratures, we generalize a recent result of Trefethen [\textit{SIAM Rev.}, 64 (2022), pp.~132--150] for the upper error bounds in terms of the decay conditions.
\end{abstract}

\begin{keywords}
  numerical integration, quadrature formula, worst-case error, analytic functions, Gauss quadrature, trapezoidal rule, Clenshaw--Curtis
\end{keywords}

\begin{AMS}
  41A55, 65D30, 65D32, 26E05, 46E15
\end{AMS}

\section{Introduction}\label{sec:intro}
Numerical integration lies at the heart of numerical analysis. 
Computing integral values of functions itself is a crucial task, for example when computing expected values of random variables;  numerical integration also serves as a foundation of other tasks in numerical analysis, such as time-marching schemes for differential equations, spatial integration when solving partial differential equations, and function approximations.

Quadrature formulas for one-dimensional integrals form the foundation of numerical integration. 
For multidimensional problems with moderate dimensions, direct product methods based on one-dimensional quadrature are often employed.
Computing one-dimensional integrals is of interest for its own sake: a large number of applications require one-dimensional integration as a subroutine: computation of matrix functions \cite{hale2008computing, TATSUOKA2020112396}, problems arising in physics \cite{Robert_TanhSinh}, and subroutines for one-dimensional integrals appearing in high-dimensional problems  \cite{griebel2014dimension,GKLS18,Nichols.J_Kuo_2014_POD}. 
We also refer to \cite[II. Applications]{Tre19} for other applications such as rational approximation, Laplace transform, and numerical solutions of partial differential equations. 
Furthermore, one-dimensional quadrature serves as a foundational element for developing and analyzing algorithms for high-dimensional problems \cite{Bungartz.H_Griebel_2004_sparse_grid_acta,DILP18,Dick.J_Kuo_Sloan_2013_ActaNumerica}.

Among the theories of quadrature, error analysis is of extreme practical importance; it enables practitioners to choose a numerical integration rule that is likely to yield a small error. 
Although the importance of error analysis is widely acknowledged, the sharpness of the error bounds appears to receive less attention. This is unfortunate, as less-than-optimal error analyses could misguide users. 
The aim of this paper is to study the sharpness of upper bounds by proving (matching) lower bounds---the bounds often more difficult to establish than upper bounds.
More precisely, we focus on a class of analytic functions defined over the real line and examine the sharpness of the worst-case error bounds for various quadrature formulas.

Our focus is on the integration problem
\begin{equation}\label{eq:integral_real_line}
    I(f):=\int_{-\infty}^{\infty}f(x)\rd x,
\end{equation}
where $f$ is an analytic function defined over the real line.
Integrals over the whole real line are important for their own sake as mentioned above. Furthermore, they form a foundation for quadrature formulas possibly on other intervals through variable transformations.
For an integral 
\(
\int_{a}^{b} g(t) \rd t
\)
with $-\infty \leq a < b \leq \infty$, 
by applying a variable transformation $t = \phi(x)$ with a function $\phi: (-\infty, \infty) \to (a,b)$, 
we have the equivalent integral 
\(
\int_{-\infty}^{\infty} g(\phi(x))\phi'(x) \rd x. 
\)
If the integrand $g(\phi(\cdot))\phi'(\cdot)$ decays rapidly at infinity, 
quadrature formulas on $(-\infty, \infty)$ will work well for computing the integral. 
In particular, 
even if $g$ has end-point singularities, 
they are annihilated by $\phi$ and do not deteriorate the computation. 
Based on this simple principle, 
single-exponential \cite{stenger1993numerical} and 
double-exponential \cite{TM1974}
quadrature formulas have been established.
This observation implies that numerical integration over the real line plays an important role even in approximating the integrals over bounded intervals, such as $[0,1]$ and $[-1,1]$.
Taking these facts into account, 
we focus on integrals on $(-\infty, \infty)$ in this paper. 

The trapezoidal rule is a popular quadrature rule for the integration problem~\eqref{eq:integral_real_line}. 
Indeed, the double-exponential formula was motivated by high efficiency of the trapezoidal rule for integrals on $(-\infty,\infty)$, which had been already observed and studied in \cite{goodwin1949evaluation,Sch1969,Ste1973}. 
Sugihara \cite{Sug97} established that the trapezoidal rule is close to optimal in the sense that is made precise below. This paper is partly motivated by improving the results in \cite{Sug97}.
Another widely used class of quadrature rules for \eqref{eq:integral_real_line} is the Gauss--type rules. 
Among these, the Gauss–Hermite formula is particularly widely used, and is also crucial for high-dimensional problems \cite{chen2017hessian,kamilis2018uncertainty,piazzola2022sparse}. 
Recently, however,
despite its optimality in terms of algebraic degree of exactness, Trefethen \cite{Tre22} pointed out its inefficiency compared with several other quadrature rules, such as the trapezoidal formula. 
Partially motivated by his result, we explain the inefficiency of the Gauss--Hermite formula in a rigorous way by proving the lower bounds of its integration error. 
Furthermore, following Trefethen \cite{Tre22}, we consider various standard quadrature formulas for finite intervals, which are applied upon truncating $(-\infty,\infty)$ to a finite interval. Our results show that this strategy is also nearly optimal.

With this aim in mind, instead of considering specific quadrature rules mentioned above, we keep the discussion general by considering a class of quadrature rules of the form
\begin{align}\label{eq:general_quadrature}
    A_n(f)=\sum_{i=1}^{n}w_if(\xi_i),
\end{align}
with a set of $n$ nodes $\{\xi_1,\dots,\xi_n\}\subset \RR$ and   corresponding  weights $\{w_1,\dots,w_n\}\subset \RR$. 
We will develop theory for this general class, which in turn gives results for the aforementioned quadrature rules as corollaries.

Any quadrature rule should aim for a small integration error $
|A_n(f)-I(f)|$, and a rapid decay of the error in $n$ when $f$ is smooth.  
One often seeks a guarantee on the error for a class of functions $F$,  
which is achieved by evaluating the \emph{worst-case error}. 
Let $F$ be a normed space. 
Then, the worst-case error in $F$ is given by
\[ e^{\wor}(A_n, F) := \sup_{\substack{f\in F\\ \|f\|_F\leq 1}}\left| A_n(f)-I(f)\right|,\]
where $\|\cdot\|_F$ denotes the norm of $F$. 
The best achievable quantity is given by the $n$-th minimal worst-case error:
\[ e^{\wor}(n, F) := \inf_{A_n}e^{\wor}(A_n, F), \]
where the infimum runs over all possible algorithms $A_n$ which use $n$ function evaluations \cite{NW2008}. 

Often, the exact value of $e^{\wor}(A_n, F)$ is not known but an upper bound on $e^{\wor}(A_n, F)$ is available. 
This upper bound, denoted by $B^{\up}(A_n, F)$, implies that
for any individual $f\in F$ we have
\[ \left| A_n(f)-I(f)\right|\leq \|f\|_F\, e^{\wor}(A_n,F)\leq \|f\|_F\, B^{\up}(A_n,F),\]
meaning that the error is at most $\|f\|_F\, B^{\up}(A_n,F)$ and
decays in $n$ at least as fast as $B^{\up}(A_n,F)$. 
Hence, upper bounds on $e^{\wor}(A_n, F)$  provide a theoretical guarantee on the accuracy and are thus of great importance.
If we know that the convergence rate of $B^{\up}(A_n,F)$ is the same as that of $e^{\wor}(n, F)$, the algorithm $A_n$ is said to achieve the optimal rate of convergence in $F$. 

What often seems to be overlooked is the sharpness of these upper bounds; non-sharp upper bounds may mislead us.
Suppose that we have two different quadrature formulas,  
$A_n^{(1)}$ and $A_n^{(2)}$, each with the same number of nodes. 
Does $B^{\up}(A_n^{(1)},F)<B^{\up}(A_n^{(2)},F)$ ensure that $A_n^{(1)}$ is a better algorithm than $A_n^{(2)}$ in $F$? 
The answer is \emph{no} if there is room for improvement of the bounds. 
Ideally, one would compare the values of the worst-case error directly, but again, evaluating the worst-case error of a given quadrature formula exactly is usually very challenging.

Of great importance, then, are lower bounds on the worst-case error. Although these are often more difficult to obtain than upper bounds, they are highly valuable when available. Lower bounds help assess the sharpness of upper bounds and thus aid in selecting suitable quadrature formulas. 
Let $B^{\low}(A_n, F)$ be a lower bound.
If the convergence rates of $B^{\low}(A_n, F)$ and $B^{\up}(A_n, F)$ in $n$ coincide, the upper bound is already tight and its improvement is possible only up to a constant factor.
Furthermore, if we have $B^{\up}(A_n^{(1)}, F)<B^{\low}(A_n^{(2)}, F)$, then we can conclude that $A_n^{(1)}$ is a better algorithm than $A_n^{(2)}$ in terms of the worst-case error in $F$.

In this paper, our interest is in analytic functions with various decay rates at infinity. Therefore, as $F$ we choose 
the \emph{weighted Hardy space}, denoted by $H^{\infty}(\Dcal_{d},\omega)$. 
This is a Banach space of functions analytic in a strip region $\Dcal_{d}$ around the real line whose decay is controlled by a weight function $\omega$. 
For example, 
when we consider Gauss--Hermite quadrature, 
the space $H^{\infty}(\mathcal{D}_{d}, \omega)$ with $\omega(x) = \exp(-x^{2})$
may be a typical choice to consider. 
Furthermore, 
the functions in the weighted Hardy spaces are intended to represent integrands obtained by variable transformations, such as double-exponential transformations.
As far as the authors are aware, this function space was first introduced in the context of numerical integration by Sugihara \cite{Sug97} to analyze the double-exponential formula from \cite{TM1974},  
where he established the optimality of this formula by proving the optimality of a suitably truncated trapezoidal rule in $H^{\infty}(\Dcal_{d},\omega)$.
Among many others, essentially the same space has been investigated by Trefethen and Weideman \cite[Section~5]{TW2014} and Trefethen \cite[Section~5]{Tre22} for exponentially convergent quadrature formulas.
The weighted Hardy space has been also considered for designing quadrature formulas by optimization
\cite{SatoshiHayakawa2020, Slevinsky.R.M_Olver_2015_UseConformalMaps,tanaka2017potential, tanaka2020construction} as well as in function approximations 
\cite{hayakawa2023convergence, TOS17, TS19}.

The main contribution of this paper is two-fold: 
\begin{enumerate}
    \item We derive two general lower bounds on the worst-case error in $H^{\infty}(\Dcal_{d},\omega)$ for quadrature rules of the form  \eqref{eq:general_quadrature}. 
    \item 
    As an application, 
    we establish almost matching worst-case error bounds for various quadrature formulas: 
    suitably truncated trapezoidal rule, scaled Gauss--Legendre quadrature, scaled Clenshaw--Curtis quadrature and Gauss--Hermite quadrature.
\end{enumerate}
The second contribution above leads to three major implications. 
Firstly, our results imply that the error bounds on the worst-case error for the trapezoidal rule in \cite{Sug97} are sharp, with no room for improvement beyond a polynomial factor.  
Secondly, our findings introduce a new open problem by revealing a discrepancy between the error decay of the trapezoidal rule and Sugihara's lower bound for general numerical integration rules in \cite{Sug97}. 
This suggests that either quadrature rules with faster error decay exist, challenging the optimality of the trapezoidal rule, or that Sugihara's general lower bound can be refined. This point will be addressed later in Remark~\ref{rem:compare_with_Sugihara}.
Thirdly, our contribution implies the sub-optimality of Gauss--Hermite quadrature relative to other above-mentioned quadrature formulas, 
which strengthens an argument made for the case $\omega(x)=\exp(-x^2)$ in a  recent paper by Trefethen \cite[Section~5]{Tre22}.

This paper is partly motivated by results for functions with finite smoothness 
\cite{Curbera.F_1998_OptimalIntegrationLipschitz,GKS23,KSG22}.
In the seminal paper \cite{Curbera.F_1998_OptimalIntegrationLipschitz},  Curbera showed that, for integration with respect to the Gaussian measure,  
the worst-case error of quadratures using the zeros of Hermite polynomials---in particular the Gauss--Hermite rule---is quadratically worse than optimal quadratures for Lipschitz continuous integrands.  
The Gauss--Hermite rule is not optimal in Sobolev scales either, as shown by two of the current authors with Suzuki \cite{KSG22};  
in the $L^2$-Sobolev spaces over $\mathbb{R}$ of order $\alpha\in \NN$, where the integrability condition is with respect to the Gaussian measure, the matching upper and lower worst-case error bounds of the order $n^{-\alpha/2}$ for the Guass-Hermite quadrature have been recently shown in \cite{KSG22}. On the other hand, a suitably truncated trapezoidal rule attains the optimal rate $n^{-\alpha}$ up to a logarithmic factor. 
More recently, a randomized trapezoidal rule has been introduced in \cite{GKS23} that attains the optimal rate $n^{-\alpha-1/2}$ of the root-mean-squared error, up to a logarithmic factor, in the same spaces. 
For the important case of analytic integrands where $\alpha$ may be taken arbitrarily large, these results do not offer insights as to which quadrature rules give smaller error,
and a different argument is necessary to discuss the optimality of quadrature formulas. 
This is the aim of this paper. 

The rest of this paper is organized as follows. 
In \cref{sec:pre}, we define the weighted Hardy space $H^{\infty}(\Dcal_{d},\omega)$, a space of analytic functions around the real line on the complex plane, and introduce the two different weight functions $\omega$. 
We also discuss the upper bounds on the worst-case error for various quadrature formulas in $H^{\infty}(\Dcal_{d},\omega)$, some of which are known in the literature \cite{Sug97,Tre22}, while the others are new in this paper. Then, as the main part of this paper, we study lower bounds on the worst-case error in $H^{\infty}(\Dcal_{d},\omega)$. In particular, we derive two different forms of a lower bound in \cref{sec:main}, and by applying these forms, we show nearly matching lower bounds for various quadrature formulas in \cref{sec:applications}, establishing the sub-optimality of the Guass-Hermite quadrature relative to the other quadratures.

\section{Preliminaries}\label{sec:pre}

\subsection{Function spaces}\label{subsec:function_spaces}
Following \cite{Sug97, TS19}, we introduce a space of analytic functions that we work with throughout this paper. 
For a real number $d>0$, we define a strip region $\Dcal_d$ around the real line on the complex plane by
\[
\Dcal_{d}:=\{z\in\CC\mid|\rIm\,z|<d\}.
\]
As considered in \cite{stenger1993numerical}, the space $B(\Dcal_d)$ is a space of analytic functions over $\Dcal_d$ such that for any function $\omega\in B(\Dcal_d)$ two conditions
\[ \lim_{x\to \pm \infty}\int_{-d}^{d}|\omega(x+\ri y)| \rd y =0 \]
and 
\[ \lim_{y\nearrow d}\int_{-\infty}^{\infty}\left(|\omega(x+\ri y)| + |\omega(x-\ri y)|\right)\rd x <\infty \]
hold.
Then, we call $\omega: \Dcal_d\to \CC$ a \emph{weight function} if 
$\omega$ 
satisfies the following: $\omega$ is in $B(\Dcal_d)$; $\omega$ does not vanish at any point in $\Dcal_d$; and $\omega$ takes positive real values in $(0,1]$ on the real line.
Moreover, for the sake of simplicity, we assume that $\omega$ is even on the real line and monotonically decreases as $x$ goes away from $0$.

Throughout this paper, we assume that the target integrand is in the following function space consisting of analytic functions with a  decay condition specified by $\omega$.
\begin{definition}
    For a weight function $\omega \in B(\Dcal_d)$, define 
    \[ H^{\infty}(\Dcal_{d},\omega):=\left\{ f:\Dcal_{d}\to\CC\;\bigg|\;\text{\ensuremath{f(z)} is analytic in \ensuremath{\Dcal_{d}} and }\|f\|:=\sup_{z\in\Dcal_{d}}\left|\frac{f(z)}{\omega(z)}\right|<\infty\right\} ,\]
    which is called the weighted Hardy space.
\end{definition}
Notice that
$f\in H^{\infty}(\Dcal_{d},\omega)$ implies $|f(x)|\leq \omega(x)\|f\|$ for $x\in \RR$, i.e., the weight function $w$ controls the decay of $f$ on $\RR$.
Regarding a lower bound on the worst-case error in $H^{\infty}(\Dcal_{d},\omega)$, the following result is implied from \cite[Eqs.~(3.11) and (3.13)]{Sug97}. We provide the proof for completeness.

\begin{proposition}\label{prop:worst-case_error}
Let $A_n$ be a quadrature rule of the form \eqref{eq:general_quadrature}. For any weight function $\omega$, the following holds:
\[ e^{\wor}(A_n, H^{\infty}(\Dcal_{d},\omega)) \geq \int_{-\infty}^{\infty}\omega(x)\prod_{i=1}^{n}\left|\tanh\left(\frac{\pi}{4d}(x-\xi_{i})\right)\right|^{2} \rd x. \]
\end{proposition}

\begin{proof}
    Let us consider a function $g: \Dcal_d\to \CC$ defined by
    \[ g(z)=\omega(z)\prod_{i=1}^{n}\tanh\left(\frac{\pi}{4d}(z-\xi_{i})\right)\overline{\tanh\left(\frac{\pi}{4d}(\overline{z}-\xi_{i})\right)}. \]
    By noting that the complex function $z\mapsto \tanh(\pi (z-\xi)/(4d))$ is a conformal mapping from $\Dcal_d$ to the unit disc $\{z\in \CC\mid |z|< 1\}$ for any $\xi\in \RR$,  it is easy to check that $g$ is analytic in $\Dcal_d$ and 
    \begin{align*}
        \|g\| & =\sup_{z\in\Dcal_{d}}\prod_{i=1}^{n}\left|\tanh\left(\frac{\pi}{4d}(z-\xi_{i})\right)\overline{\tanh\left(\frac{\pi}{4d}(\overline{z}-\xi_{i})\right)}\right| \\
        & \leq \prod_{i=1}^{n}\left(\sup_{z\in\Dcal_{d}}\left|\tanh\left(\frac{\pi}{4d}(z-\xi_{i})\right)\right|\right)^2= 1.
    \end{align*}
    This means that $g$ belongs to the unit ball in  $H^{\infty}(\Dcal_{d},\omega)$. As $g(\xi_i)=0$ for all $i=1,\ldots,n$, we have $A_n(g)=0$ and
    \begin{align*}
        e^{\wor}(A_n, H^{\infty}(\Dcal_{d},\omega)) & \geq \left| A_n(g)-I(g)\right| = I(g) \\
        & = \int_{-\infty}^{\infty}\omega(x)\prod_{i=1}^{n}\left|\tanh\left(\frac{\pi}{4d}(x-\xi_{i})\right)\right|^{2} \rd x.
    \end{align*}
    Thus the claim is proved.
\end{proof}
The ``fooling function'' $g$ in the proof above deserves some comments.
The function $g$ is designed to vanish at the nodes  $\xi_1,\ldots,\xi_n$, ensuring $A_n(g)=0$. 
A typical approach to constructing such a fooling function is to use compactly supported functions whose support does not include any node; see, for instance, \cite{Ba1959,dick2015proof,GKS23,KSG22}. Such compactly supported functions do not work in our case; any function $f$ that is analytic in the strip region and has compact support must be necessarily zero, $f\equiv 0$.  
This is why our fooling function $g$ is given in the product form, which introduces some challenges in the subsequent analyses.
 
It is also worth mentioning that the lower bound on the worst-case error given in Proposition~\ref{prop:worst-case_error} depends only on the nodes and is irrelevant to the quadrature weights.
Hence, optimizing the quadrature weights does not improve the lower bounds obtained from the result above. 

\begin{remark}
In fact, the lower bound in Proposition~\ref{prop:worst-case_error} applies to any quadrature rule, whether it depends on the function values at the quadrature nodes $\xi_1,\ldots,\xi_n\in \mathbb{R}$ linearly or non-linearly. 
Let $A_n$ be any non-linear (or linear) quadrature rule using the function values at $\xi_1,\ldots,\xi_n$,
    i.e., $A_n\colon H^{\infty}(\Dcal_{d},\omega)\to \mathbb{R}$ is a mapping of the form $A_n(f)=\mathcal{I}_n(f(\xi_1),\ldots,f(\xi_n))$ with some mapping $\mathcal{I}_n\colon \mathbb{R}^n\to \mathbb{R}$. 
    Let $g: \Dcal_d\to \CC$ be the function given in the proof of Proposition~\ref{prop:worst-case_error}. Since we have $g(\xi_1)=\cdots=g(\xi_n)=0$, i.e., we only have the information that the $n$ function values are all equal to $0$, $A_n(g)=A_n(-g)$ holds. Noting that the functions $\pm g$ belong to the unit ball in $H^{\infty}(\Dcal_{d},\omega)$, we obtain
    \begin{align*}
        e^{\wor}(A_n, H^{\infty}(\Dcal_{d},\omega)) & \geq \max\left\{\left| A_n(g)-I(g)\right|, \left| A_n(-g)-I(-g)\right|\right\}\\
        & \geq \frac{\left| A_n(g)-I(g)\right|+\left| A_n(g)+I(g)\right|}{2}\geq |I(g)|=I(g).
    \end{align*}
    This lower bound on the worst-case error even applies to quadrature rules with adaptive selection of nodes, in which the $(i+1)$-th node $\xi_{i+1}$ is allowed to depend on the already-computed function values $f(\xi_1),\ldots,f(\xi_i)$; see \cite{novak1996power} for more details on adaptive algorithms. 
    Indeed, given such an algorithm $A^*_n$, let $\xi_1^*,\ldots,\xi_n^*\in \mathbb{R}$ be the nodes that $A^*_n$ sequentially generates for the integrand $f^*\equiv 0$. We then construct the fooling function $g^*$ corresponding to these nodes. 
    The algorithm $A^*_n$ yields the same quadrature nodes for $g^*$ as it does for $f^*$, and thus $A^*_n(g^*)=A^*_n(0)=A^*_n(-g^*)$ holds, which results in the same lower bound for the worst-case error. 
    For simplicity in presentation, we adhere to the linear algorithms of the form \eqref{eq:general_quadrature} in the rest of this paper.
  \end{remark}

For the choice of a weight function $\omega$, we consider the following two classes depending on how fast
$\omega$ decays on the real line.
\begin{definition}\label{def:weight}
    We say that the weight function $\omega$ is of 
    \begin{enumerate}
    \item single-exponential (SE) type if 
    \[ \alpha_1\exp(-(\beta|x|)^{\rho})\leq \omega(x)\leq \alpha_2\exp(-(\beta|x|)^{\rho})\]
    for $\alpha_1,\alpha_2,\beta>0$ and $\rho\geq 1$, and
    \item double-exponential (DE) type if
    \[ \alpha_1\exp(-\beta_1\exp(\gamma|x|))\leq \omega(x)\leq \alpha_2\exp(-\beta_2\exp(\gamma|x|))\]
    for $\alpha_1,\alpha_2,\beta_1,\beta_2,\gamma>0$ with $\beta_1\geq \beta_2$. 
\end{enumerate}
In what follows, weight functions of the SE and DE types are denoted by $\omega_{\SE}$ and $\omega_{\DE}$, respectively.
\end{definition}

\begin{remark}
As shown in \cite[Theorem~4.1]{Sug97}, there exists no weight function $\omega\in B(\Dcal_d)$ such that
\[ \omega(x)=O(\exp(-\beta\exp(\gamma|x|)))\quad \text{as}\quad |x|\to \infty, \]
where $\beta>0$ and $\gamma>\pi/(2d).$ Therefore, when we speak of weight functions of the DE type, we implicitly assume $\gamma\leq \pi/(2d)$ without further notice. 
This non-existence result also implies that considering weight functions up to the DE type is sufficient; any weight function that decays faster than the DE type results in bounds that are vacuously true.
\end{remark}

In \cite[Lemmas~3.5 and 3.6]{Sug97}, Sugihara applied Jensen's inequality to the lower bound in Proposition~\ref{prop:worst-case_error}, getting a universal lower bound on the worst-case error that holds for any $A_n$, for the case where the weight function is of either the SE or DE type.
To be precise, it was shown that there exists a constant $c>0$ that is independent of $n$ such that
\begin{equation}\label{eq:sugihara_lower_SE}
    e^{\wor}(A_n, H^{\infty}(\Dcal_{d},\omega_{\SE}))\geq cn^{1/(\rho+1)}\exp\left( -\left(\left(\frac{2}{\rho+1}\right)^{1/\rho} 2\pi d\beta n\right)^{\rho/(\rho+1)}\right),
\end{equation}
and
\begin{equation}\label{eq:sugihara_lower_DE}
    e^{\wor}(A_n, H^{\infty}(\Dcal_{d},\omega_{\DE}))\geq c\ln n\cdot \exp\left( -\frac{2\pi d\gamma n}{\ln(\pi d\gamma n/\beta_1)}\right),
\end{equation}
respectively.

On the one hand, these lower bounds 
\eqref{eq:sugihara_lower_SE} and \eqref{eq:sugihara_lower_DE} are powerful because $A_n$ can be \emph{any} quadrature of the form 
\eqref{eq:general_quadrature}---in fact, $A_n$ can be any non-linear numerical integration that gives zero when function values at the nodes are zero. 
On the other hand, when one is interested in a specific quadrature, they might be overly general; in the proof of \eqref{eq:sugihara_lower_SE} and \eqref{eq:sugihara_lower_DE}, no information about the distribution of the nodes is used.  
In this paper, we revisit Proposition~\ref{prop:worst-case_error} and derive sharper bounds by exploiting the distribution of quadrature nodes.

Before moving on, it is worth noting that lower bounds on the worst-case error for integrals over the interval $[-1,1]$ in suitable function classes have been extensively studied for quite some time. See, for instance, \cite{andersson1980,AB1984,Bakhvalov.NS_1967_OptimalSpeedIntegrating,MZZ2016,newman1979,Petras.K_1998_GaussianOptimalIntegration}. 
However, the primary focus has traditionally been on establishing lower bounds applicable to any general quadrature formula. 
It is much less common to prove lower bounds on the worst-case error for specific quadrature formulas, aiming to demonstrate the sharpness of upper bounds for respective formulas, as done in this paper.

\subsection{Upper bounds---known and new results}\label{subsec:upper_bounds}
We summarize upper bounds available in the literature.

Arguably, Gauss--Hermite quadrature is one of the best-known formulas for the case $\omega(x) = \exp(-x^2)$ or $\omega(x) = \exp(-x^2/2)$. 
The orthogonal polynomials corresponding to these weight functions are called the physicist's Hermite polynomials and the probabilist's Hermite polynomials, respectively. 
The widely accepted upper bound for the error, at least for the quadrature corresponding to the physicist's Hermite polynomial, appears to be $O(\exp(-cn^{1/2}))$ for some constant $c > 0$, see, for instance, \cite[p.~314]{DR84}. To the best of the authors' knowledge, error analysis for analytic functions can be traced back to Barrett's work \cite{Bar1961}. In this work, the convergence rate $O(\exp(-cn^{1/2}))$ for some positive constant $c$ was mentioned, but without an explicit proof or complete statement. 
Much later, a rigorous error analysis was made in \cite{Xiang2012} for entire functions with explicit decay conditions on $\mathbb{R}$, although relating the convergence rates obtained in this study to the rate $O(\exp(-cn^{1/2}))$ does not seem to be obvious.
Only recently, Wang and Zhang \cite{WZ2023} appear to have established explicit decay conditions on the integrand that allow for the convergence rate $O(\exp(-cn^{1/2}))$, where functions that are analytic in a strip region were considered. 
The decay condition in \cite{WZ2023} is $|x|^\sigma \exp(-x^2)$, with $\sigma\in \RR$. 
Upper error bounds for other choices of $\omega$ do not seem to exist in the literature. 

In contrast, the trapezoidal rule has been studied intensively with rigorous error analyses for analytic functions; see \cite{goodwin1949evaluation,Sug97,Tre22,TW2014} among many others. 
As mentioned in \cref{sec:intro}, with an aim to establish the optimality of the double-exponential formula, 
Sugihara proved the upper bounds on the worst-case error of the trapezoidal rule in $H^{\infty}(\Dcal_{d},\omega)$ for the weight function $\omega$ being of the SE and DE types in \cite{Sug97}. 
The upper bounds he showed in \cite{Sug97} are of $O(\exp(-(\pi d\beta n)^{\rho/(\rho+1)}))$ for the SE case, and of $O(\exp(-\pi d\gamma n/\ln(\pi d \gamma n/\beta_2)))$ for the DE case, see \cite[Theorems~3.1 and 3.2]{Sug97}, respectively. 
We show in Corollary~\ref{cor:trap} that these upper bounds cannot be improved by more than a polynomial factor.

\begin{sloppypar}
More recently, Trefethen \cite{Tre22} studied various quadrature formulas, such as scaled Gauss--Legendre, Clenshaw--Curtis, and trapezoidal quadratures, for $\omega(x) = \exp(-x^2)$.
In \cite[Theorem~5.1]{Tre22}, he showed that the integration errors for these formulas decay at the rate of $O(\exp(-cn^{2/3}))$ for some constant $c>0$. 
This rate decay is more favorable to $O(\exp(-cn^{1/2}))$, the rate expected from Gauss--Hermite quadrature. 
This rate appears to be only recently proven by Wang and Zhang \cite{WZ2023}, having been considered folklore for a long time. 
\end{sloppypar}

\begin{sloppypar}
In fact, by following his argument in \cite[Appendix]{Tre22}, we can show the upper bounds on the worst-case error for scaled Gauss–Legendre and Clenshaw–Curtis quadratures when $\omega$ is of the SE and DE types, as follows.
\end{sloppypar}
\begin{theorem}\label{thm:upper_bound_new}
    For $n\geq 2,$ let $A_n$ be either the Gauss--Legendre or the Clenshaw--Curtis quadrature  scaled from $[-1,1]$ to $[-T,T]$ with
    \begin{align} \label{eq:scaling_factor}
    T=\begin{cases} Ln^{1/(\rho+1)} & \text{for $\omega=\omega_{\SE}$,} \\ L\ln n & \text{for $\omega=\omega_{\DE}$,} \end{cases}
    \end{align}
    for a fixed $L>0$, where we assume $\gamma L\geq 1$ for the case $\omega=\omega_{\DE}$ with $\gamma$ being in Definition~\ref{def:weight}. Then, there exists a constant $C>0$, independent of $n$, such that $e^{\wor}(A_n, H^{\infty}(\Dcal_{d},\omega))$ is bounded above by $O(\exp(-Cn^{\rho/(\rho+1)}))$ for the case $\omega=\omega_{\SE}$ and $O(\exp(-Cn/\ln n))$ for the case $\omega=\omega_{\DE}$, respectively.
\end{theorem}

\begin{proof}
Let $f\in H^{\infty}(\Dcal_{d},\omega)$.
For $A_n$, the Gauss--Legendre or the Clenshaw--Curtis quadrature scaled from $[-1,1]$ to $[-T,T]$, we have
\[ \left|I(f)-A_n(f)\right| \leq \left|\int_{-T}^{T}f(x)\rd x-A_n(f)\right|+\left|\int_{-\infty}^{-T}f(x)\rd x+\int_{T}^{\infty}f(x)\rd x\right|,\]
where the first term on the right-hand side is the integration error over the truncated interval $[-T,T]$ and the second term is the domain-truncation error.

First, we work with the truncation error. For either $\omega$ under study, we have 
    \begin{align*}
        \left|\int_{-\infty}^{-T}f(x)\rd x+\int_{T}^{\infty}f(x)\rd x\right| & \leq \sup_{x\in \RR}\left|\frac{f(x)}{\omega(x)}\right|\times \left( \int_{-\infty}^{-T}\omega(x)\rd x+\int_{T}^{\infty}\omega(x)\rd x\right)\\
        & \leq \|f\|\left( \int_{-\infty}^{-T}\omega(x)\rd x+\int_{T}^{\infty}\omega(x)\rd x\right).
    \end{align*}
    For the SE case, we further have
    \begin{align*}
        \int_{-\infty}^{-T}\omega(x)\rd x+\int_{T}^{\infty}\omega(x)\rd x & \leq 2\alpha_2\int_{T}^{\infty}\exp(-(\beta x)^{\rho})\rd x\\
        & \leq \frac{2\alpha_2}{\rho \beta^{\rho}T^{\rho-1}}\int_{T}^{\infty}\rho \beta^{\rho}x^{\rho-1}\exp(-(\beta x)^{\rho})\rd x\\
        & = \frac{2\alpha_2}{\rho \beta^{\rho}T^{\rho-1}}\exp(-(\beta T)^{\rho}),
    \end{align*}
    while for the DE case we have
    \begin{align*}
        \int_{-\infty}^{-T}\omega(x)\rd x+\int_{T}^{\infty}\omega(x)\rd x & \leq 2\alpha_2\int_{T}^{\infty}\exp(-\beta_2\exp(\gamma x))\rd x\\
        & \leq \frac{2\alpha_2}{\beta_2 \gamma \rme^{\gamma T}}\int_{T}^{\infty}\beta_2 \gamma \rme^{\gamma x}\exp(-\beta_2\exp(\gamma x))\rd x\\
        & = \frac{2\alpha_2}{\beta_2 \gamma \rme^{\gamma T}}\exp(-\beta_2\exp(\gamma T)),
    \end{align*}
where $\alpha_1,\alpha_2,\beta_1,\beta_2$ and $\gamma$ are in Definition \ref{def:weight}.

For the integration error, as noted in \cite[Appendix]{Tre22}, rescaling from $[-T,T]$ to $[-1,1]$ leads to an integrand (defined on the complex plane) bounded and analytic in the strip region $D_{d/T},$ which contains the Bernstein ellipses $\Ecal_r$ with foci $\pm 1$ and any parameter $r>0$ such that $r+r^{-1}\leq 2\sqrt{1+(d/T)^2}$.  
We take $r=1+d/T$, which satisfies this condition. 
According to \cite[Theorem~19.3]{Tre19}, the integration error is bounded above by $O(r^{-n})$
for Clenshaw--Curtis quadrature and $O(r^{-2n})$ for Gauss--Legendre quadrature, and thus $O(r^{-n})$ for both.
    
Now, for the SE case, by choosing $T=Ln^{1/(\rho+1)}$ for a fixed $L$, it is obvious that the truncation error is bounded above by $O(\exp(-Cn^{\rho/(\rho+1)}))$ for a constant $C>0$, while the integration error is bounded by a quantity of
\[ O(r^{-n}) = O\left(\left( 1+\frac{d}{Ln^{1/(\rho+1)}}\right)^{-Ln^{1/(\rho+1)}\times n^{\rho/(\rho+1)}/L}\right) = O\left(\exp(-C'n^{\rho/(\rho+1)})\right), \]
for a constant $C'>0$.
Similarly, for the DE case, by choosing $T=L\ln n$ for a fixed $L\geq 1/\gamma$, the truncation error is bounded above by $O(\exp(-Cn))$, while the integration error is bounded by a quantity of
\[ O(r^{-n}) = O\left(\left( 1+\frac{d}{L\ln n}\right)^{-L\ln n\times n/(L\ln n)}\right) = O\left(\exp(-C'n/\ln n)\right). \]
Thus we are done.
\end{proof}
It turns out that these rates are near optimal; see Corollary~\ref{cor:ultraspherical} and Table~\ref{tab:summary_of_results}.

\section{Main results I---general lower bounds}
\label{sec:main}
In this section, we derive two lower bounds on the worst-case error in $H^{\infty}(\Dcal_{d},\omega)$ for general quadrature rules of the form \eqref{eq:general_quadrature}.
Our starting point is Proposition~\ref{prop:worst-case_error}, which is by Sugihara \cite{Sug97}. 
His subsequent analysis following this proposition focuses on deriving a lower bound for general quadrature rules. 
In contrast, our strategy is to revisit and adapt his argument, aiming to derive different bounds for various specific quadrature rules, though in this section the bounds still hold for general quadrature rules. 
Instead of applying Jensen's inequality to the inequality in Proposition~\ref{prop:worst-case_error} as Sugihara did in \cite{Sug97}, we explore a different method, which leads us to two new bounds.

We assume that $n$ nodes in \eqref{eq:general_quadrature} are distinct and, without loss of generality, they are ordered, i.e., 
\begin{equation}
-\infty<\xi_{1}<\xi_{2}<\dots<\xi_n<\infty. \label{eq:distinct-ordered}
\end{equation}
We show one of the lower bounds in the following theorem. It should be noted that, while it may appear complicated, it is no longer in the integral form as in Proposition~\ref{prop:worst-case_error}.

\begin{theorem}
\label{thm:lower_estimate} The worst-case error of
a quadrature rule $A_n$ with $n$ distinct nodes is bounded below by
\begin{align}
    & e^{\wor}(A_n, H^{\infty}(\Dcal_{d},\omega)) \label{eq:gen-lb-2} \\
    & \geq \frac{(\tanh(1))^{4}}{30}\sum_{i=1}^{n-1}\omega(\max(|\xi_{i}|,|\xi_{i+1}|))(\xi_{i+1}-\xi_{i})\min\biggl\{1,\biggl(\frac{\pi(\xi_{i+1}-\xi_{i})}{4d}\biggr)^{4}\biggr\} \notag \\
    & \quad \times \Biggl(\prod_{j=1}^{i-1}\left|\tanh\left(\frac{\pi}{4d}(\xi_{i}-\xi_{j})\right)\right|^{2}\Biggr)\Biggl(\prod_{j'=i+2}^{n}\left|\tanh\left(\frac{\pi}{4d}(\xi_{i+1}-\xi_{j'})\right)\right|^{2}\Biggr). \notag 
\end{align}
\end{theorem}
\begin{proof}
For the integral shown on the right-hand side of Proposition~\ref{prop:worst-case_error}, cutting off the integration domain from $\RR$ to $[\xi_{1},\xi_{n}]$
and considering the monotonicity of the function $\tanh\colon\RR\to(-1,1)$ yield
\begin{align}
 & e^{\wor}(A_n, H^{\infty}(\Dcal_{d},\omega)) \label{eq:before-approx-tanh-with-x} \\
 & \geq \sum_{i=1}^{n-1}\int_{\xi_{i}}^{\xi_{i+1}}\omega(x)\prod_{j=1}^{n}\left|\tanh\left(\frac{\pi}{4d}(x-\xi_{j})\right)\right|^{2}\rd x \notag \\
 & \geq\sum_{i=1}^{n-1}\Biggl(\prod_{j=1}^{i-1}\left|\tanh\left(\frac{\pi}{4d}(\xi_{i}-\xi_{j})\right)\right|^{2}\Biggr)\Biggl(\prod_{j'=i+2}^{n}\left|\tanh\left(\frac{\pi}{4d}(\xi_{i+1}-\xi_{j'})\right)\right|^{2}\Biggr) \notag \\
 & \qquad\times\int_{\xi_{i}}^{\xi_{i+1}}\omega(x)\left|\tanh\left(\frac{\pi}{4d}(x-\xi_{i})\right)\right|^{2}\left|\tanh\left(\frac{\pi}{4d}(x-\xi_{i+1})\right)\right|^{2}\rd x. \notag
\end{align}
As the weight function $\omega$ is assumed to have a single peak at $x=0$ and to monotonically decrease as $x$ goes away from $0$, the integral on the right-most side of \eqref{eq:before-approx-tanh-with-x} for each $i=1,\dots,n$ can be bounded below by 
\begin{align*}
 & \int_{\xi_{i}}^{\xi_{i+1}}\omega(x)\left|\tanh\left(\frac{\pi}{4d}(x-\xi_{i})\right)\right|^{2}\left|\tanh\left(\frac{\pi}{4d}(\xi_{i+1}-x)\right)\right|^{2}\rd x\\
 & \quad \geq \omega(\max(|\xi_{i}|,|\xi_{i+1}|))\int_{\xi_{i}}^{\xi_{i+1}}\left|\tanh\left(\frac{\pi}{4d}(x-\xi_{i})\right)\right|^{2}\left|\tanh\left(\frac{\pi}{4d}(\xi_{i+1}-x)\right)\right|^{2}\rd x\\
 & \quad = \omega(\max(|\xi_{i}|,|\xi_{i+1}|))(\xi_{i+1}-\xi_{i})\\
 & \quad\qquad\times\int_{0}^{1}\left|\tanh\left(\frac{\pi(\xi_{i+1}-\xi_{i})}{4d}x\right)\right|^{2}\left|\tanh\left(\frac{\pi(\xi_{i+1}-\xi_{i})}{4d}(1-x)\right)\right|^{2}\rd x.
\end{align*}
To bound this quantity further, we note that 
$\tanh(s)$ for $s>0$ is an increasing and concave function with $\tanh(0)=0$ to see 
\begin{align*}
\tanh(at)
    &\geq\tanh(\min\{1,t\}\,a)
    \geq (1-\min\{1,t\})\tanh(0)+\min\{1,t\}\tanh(a)\\
    &=\tanh(a)\min\{1,t\}\qquad\text{for any}\ t\geq0\text{ and }a>0.
\end{align*}
We apply this inequality twice to the bound above to obtain
\begin{align}
 & \int_{\xi_{i}}^{\xi_{i+1}}\omega(x)\left|\tanh\left(\frac{\pi}{4d}(x-\xi_{i})\right)\right|^{2}\left|\tanh\left(\frac{\pi}{4d}(\xi_{i+1}-x)\right)\right|^{2}\rd x \label{eq:int-lb} \\
 & \quad \geq \omega(\max(|\xi_{i}|,|\xi_{i+1}|))(\xi_{i+1}-\xi_{i})\left|\tanh\left(\frac{\pi(\xi_{i+1}-\xi_{i})}{4d}\right)\right|^{4}\int_{0}^{1}x^{2}(1-x)^{2}\rd x\notag \\
 & \quad \geq \frac{(\tanh(1))^{4}}{30}\omega(\max(|\xi_{i}|,|\xi_{i+1}|))(\xi_{i+1}-\xi_{i})\min\biggl\{1,\biggl(\frac{\pi(\xi_{i+1}-\xi_{i})}{4d}\biggr)^{4}\biggr\}.\notag
\end{align}
Applying this bound to the right-most side of \eqref{eq:before-approx-tanh-with-x} proves the claim.
\end{proof}

It will turn out that, for quadrature formulas whose nodes are equispaced or concentrated more around the origin, such as the trapezoidal rule and Gauss--Hermite quadrature, the minimum separation distance between the adjacent nodes, defined by
\[
\xi_{\min}:=\min_{i=1,\ldots,n-1}(\xi_{i+1}-\xi_{i}),
\]
dominates the rate of convergence. Using this notion, the lower bound on the worst-case error in Theorem~\ref{thm:lower_estimate} can be simplified.
In the following theorem, we show the second lower bound on the worst-case error.

\begin{theorem}\label{thm:lower_estimate1} 
The worst-case error of
a quadrature rule $A_n$ with $n$ distinct nodes in $H^{\infty}(\Dcal_{d},\omega)$ is bounded below by 
\begin{align}
    & e^{\wor}(A_n, H^{\infty}(\Dcal_{d},\omega)) \label{eq:gen-lb-1} \\ 
    & \geq \frac{(\tanh(1))^{4}}{30}\xi_{\min}\min\biggl\{1,\biggl(\frac{\pi\xi_{\min}}{4d}\biggr)^{4}\biggr\}\exp\left(-\frac{2\pi d}{\xi_{\min}}\right)\sum_{i=1}^{n-1}\omega(\max(|\xi_{i}|,|\xi_{i+1}|)), \notag
\end{align}
where
$\xi_{\min}$ is the minimum separation distance of the nodes for $A_n$.
\end{theorem}

\begin{proof}
By the definition of the minimum separation distance $\xi_{\min}$, the lower estimate \eqref{eq:int-lb} is further bounded below by
\begin{align}
 & \int_{\xi_{i}}^{\xi_{i+1}}\omega(x)\left|\tanh\left(\frac{\pi}{4d}(x-\xi_{i})\right)\right|^{2}\left|\tanh\left(\frac{\pi}{4d}(\xi_{i+1}-x)\right)\right|^{2}\rd x\label{eq:int-lb-one} \\
 & \quad\geq\frac{(\tanh(1))^{4}}{30}\omega(\max(|\xi_{i}|,|\xi_{i+1}|))\xi_{\min}\min\biggl\{1,\biggl(\frac{\pi\xi_{\min}}{4d}\biggr)^{4}\biggr\}.\notag
\end{align}
Noting that $\ln|\tanh(t)|<0$ is monotonically
increasing for $t>0$, the products on the right-most side of \eqref{eq:before-approx-tanh-with-x}
for each $i=1,\dots,n-1$ can be bounded below by 
\begin{align*}
 & \Biggl(\prod_{j=1}^{i-1}\left|\tanh\left(\frac{\pi}{4d}(\xi_{i}-\xi_{j})\right)\right|^{2}\Biggr)\Biggl(\prod_{j'=i+2}^{n}\left|\tanh\left(\frac{\pi}{4d}(\xi_{i+1}-\xi_{j'})\right)\right|^{2}\Biggr)\\
 & \quad \geq \Biggl(\prod_{j=1}^{i-1}\left|\tanh\left(\frac{\pi(i-j)\xi_{\min}}{4d}\right)\right|^{2}\Biggr)\Biggl(\prod_{j'=i+2}^{n}\left|\tanh\left(\frac{\pi(j'-i-1)\xi_{\min}}{4d}\right)\right|^{2}\Biggr) \\
 & \quad = \Biggl(\prod_{j=1}^{i-1}\left|\tanh\left(\frac{\pi j\xi_{\min}}{4d}\right)\right|^{2}\Biggr)\Biggl(\prod_{j'=1}^{n-i-1}\left|\tanh\left(\frac{\pi j'\xi_{\min}}{4d}\right)\right|^{2}\Biggr)\\
 & \quad=\exp\left(2\sum_{j=1}^{i-1}\ln\left|\tanh\left(\frac{\pi j\xi_{\min}}{4d}\right)\right|+2\sum_{j'=1}^{n-i-1}\ln\left|\tanh\left(\frac{\pi j'\xi_{\min}}{4d}\right)\right|\right)\\
 & \quad\geq\exp\left(2\int_{0}^{i-1}\ln\left|\tanh\left(\frac{\pi\xi_{\min}x}{4d}\right)\right|\rd x+2\int_{0}^{n-i-1}\ln\left|\tanh\left(\frac{\pi\xi_{\min}x}{4d}\right)\right|\rd x\right)\\
 & \quad=\exp\left(2\int_{-(i-1)}^{n-i-1}\ln\left|\tanh\left(\frac{\pi\xi_{\min}x}{4d}\right)\right|\rd x\right)\\
 & \quad\geq\exp\left(\frac{8d}{\pi\xi_{\min}}\int_{-\infty}^{\infty}\ln\left|\tanh(y)\right|\rd y\right)=\exp\left(-\frac{2\pi d}{\xi_{\min}}\right).
\end{align*}
Hence, together with Theorem~\ref{thm:lower_estimate} and \eqref{eq:int-lb-one}, we conclude
\begin{align*}
  & e^{\wor}(A_n, H^{\infty}(\Dcal_{d},\omega)) \\
  & \geq \int_{-\infty}^{\infty}\omega(x)\prod_{i=1}^{n}\left|\tanh\left(\frac{\pi}{4d}(x-\xi_{i})\right)\right|^{2}\rd x\\
  & \geq\frac{(\tanh(1))^{4}}{30}\xi_{\min}\min\biggl\{1,\biggl(\frac{\pi\xi_{\min}}{4d}\biggr)^{4}\biggr\}\exp\left(-\frac{2\pi d}{\xi_{\min}}\right)\sum_{i=1}^{n-1}\omega(\max(|\xi_{i}|,|\xi_{i+1}|)).
\end{align*}
This proves \eqref{eq:gen-lb-1}. 
\end{proof}
For quadrature formulas whose nodes are densely distributed away from the origin, such as scaled Gauss--Legendre quadrature and scaled Clenshaw--Curtis quadrature, the lower bound shown in Theorem~\ref{thm:lower_estimate1} does not necessarily lead to sharp results, as will be seen in Remark~\ref{rem:ultraspherical}. In such cases, we need to use the lower bound shown in Theorem~\ref{thm:lower_estimate} directly.

\section{Main results II---applications to various quadrature formulas}\label{sec:applications}
Applying Theorem~\ref{thm:lower_estimate} and Theorem~\ref{thm:lower_estimate1} to various quadrature formulas, here we show almost matching lower bounds on the worst-case error in $H^{\infty}(\Dcal_{d},\omega)$, as compared to the upper bounds given in \cref{subsec:upper_bounds}. 
Table~\ref{tab:summary_of_results} summarizes the results. 
As can be seen, except for Gauss--Hermite quadrature for the DE case where an upper bound is unknown, respective lower and upper bounds on the worst-case error almost coincide with each other. The sharpness of our lower bounds is quite clear for the trapezoidal rule for which even the factors in the ``main'' exponential terms for the lower and upper bounds coincide, and moreover, as we prove in Corollary~\ref{cor:trap}, these factors are given explicitly. Such precise estimates are not available for scaled Gauss--Legendre and Clenshaw--Curtis quadratures. The rate of the lower bound for Gauss--Hermite quadrature is worse than the rates of the upper bounds for other quadratures (except for the SE case with $\rho=1$), which concludes the sub-optimality of Gauss--Hermite quadrature. 

\begin{table}[t]
\caption{Summary of the lower and upper bounds of the worst-case errors. 
The real numbers $c$, $c'$, $C$, and $C'$ are positive and independent of $n$.
They are different for each row. 
The sequences $p_{n}$ and $q_{n}$ are given by
$p_{n} = n^{\rho/(\rho+1)}$ and $q_{n} = n/\ln n$, respectively. In the leftmost column, \emph{trapz}, \emph{G-L}, \emph{C-C}, and \emph{G-H} denote suitably truncated trapezoidal rule, scaled Gauss--Legendre quadrature, scaled Clenshaw--Curtis quadrature, and Gauss--Hermite quadrature, respectively.}
\label{tab:summary_of_results}
\centering
   \begin{tabular}{p{0.07\linewidth}|p{0.23\linewidth}|p{0.17\linewidth}|p{0.17\linewidth}|p{0.15\linewidth}|}
        & \multicolumn{2}{c|}{SE} & \multicolumn{2}{c|}{DE} \\
        & lower & upper & lower & upper \\
        \hline
        trapz 
        & \(\Omega(p_{n}^{-4} \mathrm{e}^{-c p_{n}})\) \newline Cor.~\ref{cor:trap}
        & \(O(\mathrm{e}^{-c p_{n}})\) \newline \cite{Sug97,Tre22} 
        & \(\Omega(q_{n}^{-4} \mathrm{e}^{-C q_{n}})\) \newline Cor.~\ref{cor:trap}
        & \(O(\mathrm{e}^{-C q_{n}})\) \newline \cite{Sug97} \\
        \hline
        G-L \newline C-C 
        & \(\Omega(p_{n}^{-5} \mathrm{e}^{-c p_{n}})\) \newline Cor.~\ref{cor:ultraspherical}
        & \(O(\mathrm{e}^{-c' p_{n}})\) \newline \cite{Tre22}, Thm.~\ref{thm:upper_bound_new} 
        & \(\Omega(q_{n}^{-5} \mathrm{e}^{-C q_n})\) \newline Cor.~\ref{cor:ultraspherical}
        & \(O(\mathrm{e}^{-C' q_{n}})\) \newline Thm.~\ref{thm:upper_bound_new} \\
        \hline
        G-H 
        & \(\Omega(n^{-2} \mathrm{e}^{-c \sqrt{n}})\) \newline 
        Cor.~\ref{cor:gauss-hermite-general} %
        & \(O(\mathrm{e}^{-c' \sqrt{n}})\) \newline \cite{Bar1961,WZ2023} (\(\rho=2\))
        & \(\Omega(n^{-2} \mathrm{e}^{-C \sqrt{n}})\) \newline 
        Cor.~\ref{cor:gauss-hermite-general}%
        & --- \\
        \hline
    \end{tabular}
\end{table}

\subsection{Trapezoidal rule}

First, let us consider the equispaced $n=2m+1$ nodes distributed around
0. We write 
\[
\xi_{i}=i\xi_{\min},\quad\text{for all \ensuremath{i=-m,\ldots,m},}
\]
with 
\[
\xi_{\min}=\begin{cases}
{\displaystyle (2\pi d)^{1/(\rho+1)}(\beta m)^{-\rho/(\rho+1)}} & \text{for $\omega=\omega_{\SE}$,}\\
{\displaystyle \ln(2\pi d\gamma m/\beta_2)/(\gamma m)} & \text{for $\omega=\omega_{\DE}$.}
\end{cases}
\]
For these formulas for $\xi_{\mathrm{min}}$, see, e.g., \cite{Sug97}.
Note that here we employ a different numbering for the nodes, i.e., $\xi_{-m},\ldots,\xi_0,\ldots,\xi_m$, instead of $\xi_1,\ldots,\xi_n$ that have been used so far. 
According to \cite[Theorems~3.1 and 3.2]{Sug97}, the corresponding trapezoidal rule is
\begin{align}\label{eq:trap_rule}
    A_n(f)=\xi_{\min}\sum_{i=-m}^{m} f(\xi_i).
\end{align}

As we mentioned in \cref{subsec:upper_bounds}, upper bounds for this rule are known \cite[Theorems 3.1 and 3.2]{Sug97}, while lower bounds \eqref{eq:sugihara_lower_SE} and 
\eqref{eq:sugihara_lower_DE} for general quadrature rules are also known. 
Although these bounds are close, there is a slight gap in the rates; the question of whether the trapezoidal rule attains the general lower bounds \eqref{eq:sugihara_lower_SE} and 
\eqref{eq:sugihara_lower_DE}, or if the general bounds can be improved, has remained an open problem. 
The following result establishes that the upper bounds on the worst-case error for the trapezoidal rule in $H^{\infty}(\Dcal_{d},\omega)$ by Sugihara are sharp and not improvable more than a polynomial factor.

\begin{corollary}\label{cor:trap}
The worst-case error of the trapezoidal rule with $n=2m+1$ nodes in $H^{\infty}(\Dcal_{d},\omega)$ is bounded below by 
\[
\frac{\alpha_{1}c_{\beta,\rho}(\tanh(1))^{4}}{15}\left(\frac{\pi}{4d}\cdot\frac{(2\pi d)^{1/(\rho+1)}}{(\beta n)^{\rho/(\rho+1)}}\right)^{4}\exp\left(-(\pi d\beta n)^{\rho/(\rho+1)}\right),
\]
for the case $\omega=\omega_{\SE}$, and 
\[
\frac{\alpha_{1}c_{\beta_{1},\beta_{2},\gamma}(\tanh(1))^{4}}{15}\left(\frac{\pi}{4d}\cdot\frac{\ln(\pi d\gamma n/\beta_{2})}{\gamma n}\right)^{4}\exp\left(-\frac{\pi d\gamma n}{\ln(\pi d\gamma n/\beta_{2})}\right),
\]
for the case $\omega=\omega_{\DE}$, where the constants $c_{\beta,\rho},c_{\beta_{1},\beta_{2},\gamma}>0$
do not depend on $n$. 
\end{corollary}
\begin{proof}
We apply Theorem~\ref{thm:lower_estimate1}. 
Let us consider the case $\omega=\omega_{\SE}$ first. 
Because of the symmetry of nodes, we have 
\begin{align*}
\sum_{i=-m}^{m-1}\omega(\max(|\xi_{i}|,|\xi_{i+1}|)) & \geq2\alpha_{1}\sum_{i=1}^{m}\rme^{-(\beta i\xi_{\min})^{\rho}}\geq2\alpha_{1}\int_{1}^{m+1}\rme^{-(\beta x\xi_{\min})^{\rho}}\rd x\\
 & =\frac{2\alpha_{1}}{\xi_{\min}}\int_{\xi_{\min}}^{(m+1)\xi_{\min}}\rme^{-(\beta x)^{\rho}}\rd x.
\end{align*}
When $m\geq\rho$, the integration domain $[\xi_{\min},(m+1)\xi_{\min}]$
contains a sub-interval 
\[
[(2\pi d/\beta^{\rho})^{1/(\rho+1)},(2\pi d/\beta^{\rho})^{1/(\rho+1)}(\rho+1)/\rho^{\rho/(\rho+1)}],
\]
where we used the fact that $(x+1)/x^{\rho/(\rho+1)}$ for $x>0$
is minimized when $x=\rho$ to get the upper limit. Therefore, we
have 
\begin{align*}
\sum_{i=-m}^{m-1}\omega(\max(|\xi_{i}|,|\xi_{i+1}|)) & \geq\frac{2\alpha_{1}}{\xi_{\min}}\int_{(2\pi d/\beta^{\rho})^{1/(\rho+1)}}^{(2\pi d/\beta^{\rho})^{1/(\rho+1)}(\rho+1)/\rho^{\rho/(\rho+1)}}\rme^{-(\beta x)^{\rho}}\rd x=:\frac{2\alpha_{1}c_{\beta,\rho}}{\xi_{\min}},
\end{align*}
so that the worst-case error of the trapezoidal rule with large $n$ such that $\pi\xi_{\min}\leq4d$ is bounded below by 
\begin{align*}
 & e^{\wor}(A_n, H^{\infty}(\Dcal_{d},\omega_{\SE}))
 \\
 & \quad\geq\frac{\alpha_{1}c_{\beta,\rho}(\tanh(1))^{4}}{15}\left(\frac{\pi\xi_{\min}}{4d}\right)^{4}\exp\left(-\frac{2\pi d}{\xi_{\min}}\right)\\
 & \quad=\frac{\alpha_{1}c_{\beta,\rho}(\tanh(1))^{4}}{15}\left(\frac{\pi}{4d}\cdot\frac{(2\pi d)^{1/(\rho+1)}}{(\beta m)^{\rho/(\rho+1)}}\right)^{4}\exp\left(-(2\pi d\beta m)^{\rho/(\rho+1)}\right).
\end{align*}
Since a similar lower bound can be proven for small $n$ by adjusting the constant $c_{\beta,\rho}$, by noting $n=2m+1$, this proves the result for the case $\omega=\omega_{\SE}$.

Let us move on to the case $\omega=\omega_{\DE}$. 
Again, because of the symmetry of nodes, we have 
\begin{align*}
\sum_{i=-m}^{m-1}\omega(\max(|\xi_{i}|,|\xi_{i+1}|)) & \geq 2\alpha_{1}\sum_{i=1}^{m}\rme^{-\beta_{1}\rme^{\gamma i\xi_{\min}}}\geq 2\alpha_{1}\int_{1}^{m+1}\rme^{-\beta_{1}\rme^{\gamma x\xi_{\min}}}\rd x\\
 & =\frac{2\alpha_{1}}{\xi_{\min}}\int_{\xi_{\min}}^{(m+1)\xi_{\min}}\rme^{-\beta_{1}\rme^{\gamma x}}\rd x.
\end{align*}
The lower endpoint of the integration domain $[\xi_{\min},(n+1)\xi_{\min}]$
is bounded above by 
\[
\xi_{\min}=\frac{\ln(2\pi d\gamma m/\beta_{2})}{\gamma m}\leq\frac{2\pi d}{\beta_{2}\rme},
\]
where the equality holds if $m=\beta_{2}\rme/(2\pi d\gamma)$. 
The upper endpoint is given by 
\[
(m+1)\xi_{\min}=\frac{\ln(2\pi d\gamma/\beta_{2})+\ln(m)}{\gamma}\left(1+\frac{1}{m}\right).
\]
Since there exists $m_{0}\in\NN$, which depends on $\beta_{2}$ and
$\gamma$, such that $(m+1)\xi_{\min}$ monotonically increases toward infinity for $m\geq m_{0}$, for $m$ large enough the upper endpoint is bounded below by $2\pi d/(\beta_{2}\rme)+1$. 
This way, again for
$m$ large enough, the integration domain $[\xi_{\min},(m+1)\xi_{\min}]$
contains a sub-interval 
\[
\left[\frac{2\pi d}{\beta_{2}\rme},\frac{2\pi d}{\beta_{2}\rme}+1\right].
\]
Therefore, we have 
\[
\sum_{i=-m}^{m}\omega(\max(|\xi_{i}|,|\xi_{i+1}|))\geq\frac{2\alpha_{1}}{\xi_{\min}}\int_{2\pi d/(\beta_{2}\rme)}^{2\pi d/(\beta_{2}\rme)+1}\rme^{-\beta_{1}\rme^{\gamma x}}\rd x:=\frac{2\alpha_{1}c_{\beta_{1},\beta_{2},\gamma}}{\xi_{\min}},
\]
so that the worst-case error of the trapezoidal rule with large $n$
is bounded below by 
\begin{align*}
 & e^{\wor}(A_n, H^{\infty}(\Dcal_{d},\omega_{\DE}))\\
 & \quad\geq\frac{\alpha_{1}c_{\beta_{1},\beta_{2},\gamma}(\tanh(1))^{4}}{15}\left(\frac{\pi\xi_{\min}}{4d}\right)^{4}\exp\left(-\frac{2\pi d}{\xi_{\min}}\right)\\
 & \quad\geq\frac{\alpha_{1}c_{\beta_{1},\beta_{2},\gamma}(\tanh(1))^{4}}{15}\left(\frac{\pi}{4d}\cdot\frac{\ln(2\pi d\gamma m/\beta_{2})}{\gamma m}\right)^{4}\exp\left(-\frac{2\pi d\gamma m}{\ln(2\pi d\gamma m/\beta_{2})}\right).
\end{align*}
By noting $n=2m+1$, this proves the result for the case $\omega=\omega_{\DE}$.
\end{proof}

\begin{remark}\label{rem:compare_with_Sugihara}
    As mentioned in \cref{subsec:upper_bounds}, Sugihara \cite{Sug97} proved upper bounds on the worst-case error of the trapezoidal rule \eqref{eq:trap_rule} in $H^{\infty}(\Dcal_{d},\omega)$, which are of $O(\exp(-\pi d\beta n)^{\rho/(\rho+1)})$ for the SE case, and of $O(\exp(-\pi d\gamma n/\ln(\pi d \gamma n/\beta_2)))$ for the DE case.
    We see that our obtained lower bounds match these upper bounds, apart from polynomial factors. The exponents of the ``main'' exponential terms are improved from the universal lower bounds presented in \cite{Sug97}, see \eqref{eq:sugihara_lower_SE} and \eqref{eq:sugihara_lower_DE}, respectively. Whether it is possible to improve the universal lower bounds in \cite{Sug97} remains open for future research.
\end{remark}

\subsection{Scaled quadratures related to the ultraspherical polynomials}

The ultraspherical polynomials, or the Gegenbauer polynomials, are orthogonal polynomials on the interval $[-1,1]$ with respect to the weight $(1-x^2)^{\alpha}$ for some $\alpha> -1$.
This family of polynomials is a special case of the Jacobi polynomials, which are orthogonal with respect to the weight $(1-x)^{\alpha}(1+x)^{\beta}$ for some $\alpha,\beta>-1$. Since the ultraspherical polynomials are a generalization of the Legendre polynomials ($\alpha=0$), the Chebyshev polynomials of the first kind ($\alpha=-1/2$), and the Chebyshev polynomials of the second kind ($\alpha=1/2$),  the following result applies to any quadrature formula 
that uses the roots of such polynomials as its nodes, 
such as scaled Gauss--Legendre, scaled Gauss--Chebychev, and scaled Clenshaw--Curtis quadratures.

We will apply the results from \cref{sec:main}.
Let $\xi_1,\dots,\xi_n$ be the zeros of the ultraspherical polynomial of degree $n$, with each multiplied by a common scaling factor $T>0$. 
Then, since the quadrature nodes for the Gauss--Jacobi quadrature are distinct, they satisfy \eqref{eq:distinct-ordered}:
\[ -\infty<\xi_{1}<\xi_{2}<\dots<\xi_n<\infty,\]
and moreover, without the scaling $T$ they lie in $(-1,1)$, and thus upon scaling they lie in $(-T,T)$; see for example \cite[Theorem 3.3.1]{Szego1975}.
By denoting $\xi_i=T\cos \theta_i$ for $i=1,\ldots,n,$ we have
\[ 0 <\theta_n<\theta_{n-1}<\cdots<\theta_1<\pi,\]
which satisfy the antithetic property $\theta_i+\theta_{n-i+1}=\pi$ for all $i$. 
Therefore, we have
\[ \xi_i=T\cos\theta_i = T\cos(\pi-\theta_{n-i+1})=-T\cos\theta_{n-i+1}=-\xi_{n-i+1},\]
for all $i$.
For $-1/2\leq \alpha\leq 1/2$, the classical bounds for the $\theta_i$'s in \cite[Theorem~6.21.3]{Szego1975} tell us 
\begin{align}\label{eq:ultra_nodes1}
    \left( 1-\frac{i}{n+1}\right)\pi \leq \theta_i\leq \left( 1-\frac{i-1/2}{n}\right)\pi, \quad \text{for $i=1,\ldots,\lfloor n/2\rfloor$,}
\end{align}
and 
\begin{align}\label{eq:ultra_nodes2}
    \frac{i-1/2}{n}\pi \leq \theta_{n-i+1}\leq \frac{i}{n+1}\pi, \quad \text{for $i=1,\ldots,\lfloor n/2\rfloor$,}
\end{align}
For $n$ odd, it holds that $\theta_{(n+1)/2}=\pi/2$ and $\xi_{(n+1)/2}=0.$

We obtain the following lower bounds on the worst-case error in $H^{\infty}(\Dcal_{d},\omega)$ for any quadrature rules with nodes given by the zeros of the scaled ultraspherical polynomials.
Note that our results and the universal lower bounds (\eqref{eq:sugihara_lower_SE} and \eqref{eq:sugihara_lower_DE}) are generally not comparable. More precisely, when we set the scaling factor $T$ by following \eqref{eq:scaling_factor}, our results are tighter when $L$ is large, while the opposite holds when $L$ is small. Optimizing $L$ in terms of upper bounds on the worst-case error, as well as discussing which lower estimate is the sharpest, goes beyond the scope of this paper.

\begin{corollary}\label{cor:ultraspherical}
    For $n\geq 4,$ let $A_n$ be the Gauss--Legendre quadrature scaled from $[-1,1]$ to $[-T,T]$ with
    \[ T=\begin{cases} Ln^{1/(\rho+1)} & \text{for $\omega=\omega_{\SE}$,} \\ L\ln n & \text{for $\omega=\omega_{\DE}$,} \end{cases}\]
    for a fixed $L$ (cf.\ \eqref{eq:scaling_factor}).
    Then there exists a constant $c>0$, independent of $n$, such that $e^{\wor}(A_n, H^{\infty}(\Dcal_{d},\omega))$ is bounded below by $\Omega(\exp(-cn^{\rho/(\rho+1)})/n^{5\rho/(\rho+1)})$ for the case $\omega=\omega_{\SE}$ and $\Omega(\exp(-cn/\ln n)/(n/\ln n)^5)$ for the case $\omega=\omega_{\DE}$, respectively. 
    The same lower bound with different constants holds for the scaled Clenshaw--Curtis quadrature $A_n$ for $n\geq 6$.
\end{corollary}
\begin{proof}
    We show the result for the scaled Gauss--Legendre quadrature only. 
    The result for the scaled Clenshaw--Curtis quadrature follows by the same arguments.
    From Theorem~\ref{thm:lower_estimate}, it holds for any $A_n$ that
    \begin{align*}
    & e^{\wor}(A_n, H^{\infty}(\Dcal_{d},\omega)) \\
    & \geq \frac{(\tanh(1))^{4}}{30}\sum_{i=1}^{n-1}\omega(\max(|\xi_{i}|,|\xi_{i+1}|))(\xi_{i+1}-\xi_{i})\min\biggl\{1,\biggl(\frac{\pi(\xi_{i+1}-\xi_{i})}{4d}\biggr)^{4}\biggr\} \\
    & \quad \qquad \times \exp\left(2\sum_{j=1}^{i-1}\ln\left|\tanh\left(\frac{\pi}{4d}(\xi_{i}-\xi_{j})\right)\right|+2\sum_{j'=i+2}^{n}\ln\left|\tanh\left(\frac{\pi}{4d}(\xi_{i+1}-\xi_{j'})\right)\right|\right).
    \end{align*}
    
    Let us consider the term with $i=\lfloor n/2\rfloor$ from above. For $j=1,\ldots,\lfloor n/2\rfloor-1$, it holds that
    \[ \xi_{\lfloor n/2\rfloor}-\xi_{j} = T\cos \theta_{\lfloor n/2\rfloor}-T\cos \theta_{j}=2T\sin\frac{\theta_{\lfloor n/2\rfloor}+\theta_{j}}{2}\sin\frac{\theta_{j}-\theta_{\lfloor n/2\rfloor}}{2}. \]
    By using \eqref{eq:ultra_nodes1}, we have
    \[ \frac{\theta_{\lfloor n/2\rfloor}+\theta_{j}}{2}\geq \left( 1-\frac{\lfloor n/2\rfloor+j}{2(n+1)}\right)\pi \geq \left( 1-\frac{n}{2(n+1)}\right)\pi \geq \frac{\pi}{2}, \]
    and
    \[ \frac{\theta_{\lfloor n/2\rfloor}+\theta_{j}}{2}\leq \left( 1-\frac{\lfloor n/2\rfloor+j-1}{2n}\right)\pi\leq \left( 1-\frac{(n-1)/2}{2n}\right)\pi=\left( \frac{3}{4}+\frac{1}{4n}\right)\pi\leq \frac{7\pi}{8}, \]
    so that
    \[ \sin\frac{\theta_{\lfloor n/2\rfloor}+\theta_{j}}{2}\geq \sin\frac{7\pi}{8}=\frac{\sqrt{2-\sqrt{2}}}{2}. \]
    Similarly, it follows from \eqref{eq:ultra_nodes1} that
    \[ \frac{\theta_{j}-\theta_{\lfloor n/2\rfloor}}{2}\geq \left(\frac{\lfloor n/2\rfloor-1/2}{n}-\frac{j}{n+1}\right)\frac{\pi}{2}\geq \frac{\lfloor n/2\rfloor-1/2-j}{n}\cdot \frac{\pi}{2}>0,\]
    and
    \[ \frac{\theta_{j}-\theta_{\lfloor n/2\rfloor}}{2}\leq \left(\frac{\lfloor n/2\rfloor}{n+1}-\frac{j-1/2}{n}\right)\frac{\pi}{2}\leq \frac{(n-1)/2}{n+1}\cdot\frac{\pi}{2}\leq \frac{\pi}{4}.\]
    Using these results, the first inner sum appearing in the exponential term with $i=\lfloor n/2\rfloor$ is bounded below by
    \begin{align*}
        & \sum_{j=1}^{\lfloor n/2\rfloor-1}\ln\left|\tanh\left(\frac{\pi}{4d}(\xi_{\lfloor n/2\rfloor}-\xi_{j})\right)\right| \\
        & = \sum_{j=1}^{\lfloor n/2\rfloor-1}\ln\left|\tanh\left(\frac{\pi}{4d}2T\sin\frac{\theta_{\lfloor n/2\rfloor}+\theta_{j}}{2}\sin\frac{\theta_{j}-\theta_{\lfloor n/2\rfloor}}{2}\right)\right|\\
        & \geq \sum_{j=1}^{\lfloor n/2\rfloor-1}\ln\left|\tanh\left(\frac{T\pi\sqrt{2-\sqrt{2}}}{4d}\sin\left(\frac{\lfloor n/2\rfloor-1/2-j}{n}\cdot \frac{\pi}{2}\right)\right)\right|\\
        & = \sum_{j=1}^{\lfloor n/2\rfloor-1}\ln\left|\tanh\left(\frac{T\pi\sqrt{2-\sqrt{2}}}{4d}\sin\left(\frac{j-1/2}{n}\cdot \frac{\pi}{2}\right)\right)\right|\\
        & \geq \sum_{j=1}^{\lfloor n/2\rfloor-1}\ln\left|\tanh\left(\frac{T\pi\sqrt{2-\sqrt{2}}}{4d}\cdot \frac{j-1/2}{n}\right)\right|.
    \end{align*}
    For $n$ even, regarding the second inner sum appearing in the same exponential term, the symmetry of $\xi_i$'s gives
    \begin{align*}
        & \sum_{j'=\lfloor n/2\rfloor+2}^{n}\ln\left|\tanh\left(\frac{\pi}{4d}(\xi_{\lfloor n/2\rfloor+1}-\xi_{j'})\right)\right| \\
        & = \sum_{j'= 1}^{n/2-1}\ln\left|\tanh\left(\frac{\pi}{4d}(\xi_{n/2}-\xi_{j'})\right)\right|\geq \sum_{j=1}^{\lfloor n/2\rfloor-1}\ln\left|\tanh\left(\frac{T\pi\sqrt{2-\sqrt{2}}}{4d}\cdot \frac{j-1/2}{n}\right)\right|.
    \end{align*}
    For $n$ odd, as we know $\xi_{(n+1)/2}=0$, it follows that
    \begin{align*}
        & \sum_{j'=\lfloor n/2\rfloor+2}^{n}\ln\left|\tanh\left(\frac{\pi}{4d}(\xi_{\lfloor n/2\rfloor+1}-\xi_{j'})\right)\right| \\
        & = \sum_{j'= 1}^{(n-1)/2}\ln\left|\tanh\left(\frac{\pi}{4d}\xi_{j'}\right)\right|=\sum_{j'= 1}^{(n-1)/2}\ln\left|\tanh\left(\frac{T\pi}{4d}\cos \theta_{j'}\right)\right|\\
        & \geq \sum_{j'= 1}^{(n-1)/2}\ln\left|\tanh\left(\frac{T\pi}{4d}\cos \frac{\pi j'}{n+1} \right)\right|\geq \sum_{j'= 1}^{(n-1)/2}\ln\left|\tanh\left(\frac{T\pi}{4d}\left(1-\frac{2 j'}{n+1}\right) \right)\right|\\
        & = \sum_{j'= 1}^{(n-1)/2}\ln\left|\tanh\left(\frac{T\pi j'}{2d(n+1)} \right)\right|\geq \sum_{j'= 1}^{(n-1)/2}\ln\left|\tanh\left(\frac{T\pi\sqrt{2-\sqrt{2}}}{4d}\cdot \frac{j'-1/2}{n}\right)\right|.
    \end{align*}

    Thus, for $n$ either even or odd, we have
    \begin{align*}
        & \sum_{j=1}^{\lfloor n/2\rfloor-1}\ln\left|\tanh\left(\frac{\pi}{4d}(\xi_{\lfloor n/2\rfloor}-\xi_{j})\right)\right|+\sum_{j'=\lfloor n/2\rfloor+2}^{n}\ln\left|\tanh\left(\frac{\pi}{4d}(\xi_{\lfloor n/2\rfloor+1}-\xi_{j'})\right)\right| \\
        & \geq 2\sum_{j=1}^{\lfloor n/2\rfloor}\ln\left|\tanh\left(\frac{T\pi\sqrt{2-\sqrt{2}}}{4d}\cdot \frac{j-1/2}{n}\right)\right|\\
        & \geq 2\left(\ln\left|\tanh\left(\frac{T\pi\sqrt{2-\sqrt{2}}}{8dn}\right)\right|+\int_{0}^{\infty}\ln\left|\tanh\left(\frac{T\pi\sqrt{2-\sqrt{2}}}{4d}\cdot \frac{x}{n}\right)\right|\rd x\right)\\
        & \geq 2\left(-\frac{n\pi d}{T\sqrt{2-\sqrt{2}}}-\frac{n\pi d}{T2\sqrt{2-\sqrt{2}}}\right)=-\frac{3n\pi d}{T\sqrt{2-\sqrt{2}}}.
    \end{align*}
    Here the last inequality follows from 
    \[ t\, \ln \left|\tanh(a t)\right| \geq \int_0^{\infty}\ln \left|\tanh(a t)\right|\rd t=-\frac{\pi^2}{8a},\]
    which holds for any $t,a>0$ and is obtained by noticing that $\ln |\tanh(x)|<0$ for any $x$ and that the left-hand side represents the signed area of the rectangle $[0,t]\times [\ln |\tanh(at)|,0].$
    
    Note that, for even $n$,
    \begin{align*}
    \xi_{\lfloor n/2\rfloor+1}-\xi_{\lfloor n/2\rfloor} & = 2\xi_{n/2+1} = 2T\cos \theta_{n/2+1} \\
    & \geq 2T\cos \frac{n\pi}{2(n+1)} = 2T\sin \frac{\pi}{2(n+1)}\geq \frac{2T}{n+1},
    \end{align*}
    whereas, for odd $n$,
    \begin{align*}
    \xi_{\lfloor n/2\rfloor+1}-\xi_{\lfloor n/2\rfloor} & = -\xi_{(n-1)/2} = -T\cos \theta_{(n-1)/2} \\
    & \geq T\cos \frac{(n-1)\pi}{2(n+1)} = T\sin \frac{\pi}{n+1}\geq \frac{2T}{n+1}.
    \end{align*}
    
    Altogether we obtain
\begin{align*}
    e^{\wor}(A_n, H^{\infty}(\Dcal_{d},\omega)) & \geq \frac{(\tanh(1))^{4}}{30}\omega(|\xi_{\lfloor n/2\rfloor}|) \frac{2T}{n+1}\min\biggl\{1,\biggl(\frac{\pi T}{2d(n+1)}\biggr)^{4}\biggr\} \\
    & \quad \times \exp\left(-\frac{3n\pi d}{T\sqrt{2-\sqrt{2}}}\right).
\end{align*}
Then our choice of the scaling factor $T$ obviously leads to the claimed forms of the lower bounds on $e^{\wor}(A_n, H^{\infty}(\Dcal_{d},\omega))$ for both the SE and DE cases, respectively.

The $n$-point Clenshaw--Curtis quadrature $A_n$ uses the zeros of the degree-$(n-2)$ Chebyshev polynomials of the second kind, in addition to the endpoints. 
Hence, for $n \geq 6$, the argument discussed above is applicable to the ($n-2$) zeros in the middle, instead of $n$ zeros, and thus the lower bound is identical up to a constant.
\end{proof}

\begin{remark}\label{rem:ultraspherical}
    It is known from \cite[Theorem~6.3.3]{Szego1975} that, for any $-1/2\leq \alpha\leq 1/2$, we have
    \[ \theta_n-\theta_{n-1}\leq \theta_{n-1}-\theta_{n-2}\leq \cdots\leq \theta_{\lfloor n/2\rfloor+1}-\theta_{\lfloor n/2\rfloor}.\]
    Because of the antithetic property $\theta_i+\theta_{n-i+1}=\pi$ and the concavity of the cosine function on the interval $[0,\pi/2]$, it holds that
    \begin{align*}
        \xi_{\min} & = \xi_{n}-\xi_{n-1}=T\left( \cos\theta_n-\cos\theta_{n-1}\right)\\
        & \geq T\left( \cos\frac{\pi}{n+1}-\cos\frac{3\pi}{2n}\right) \\
        & \geq T\left( 1-\frac{1}{2}\left(\frac{\pi}{n+1}\right)^2-1+\frac{4}{\pi^2}\left(\frac{3\pi}{2n}\right)^2\right)\geq \left(9-\frac{\pi^2}{2}\right)\frac{T}{n^2},
    \end{align*}
    for any $n\geq 3$, where, from the second to third lines above, we have used the inequality $1-x^2/2\leq \cos x\leq 1-4x^2/\pi^2$ for any $0\leq x\leq \pi/2$.
    If we apply Theorem~\ref{thm:lower_estimate1}, instead of Theorem~\ref{thm:lower_estimate}, under the same choice of $T$ as in Corollary~\ref{cor:ultraspherical}, the main exponential terms of the lower bounds on $e^{\wor}(A_n, H^{\infty}(\Dcal_{d},\omega))$ result in $\rme^{-c n^{(2\rho+1)/(\rho+1)}}$ for the SE case and $\rme^{-c n^2/\ln n}$ for the DE case, respectively, for some constant $c>0$. This way, it can be confirmed that Theorem~\ref{thm:lower_estimate} leads to sharper estimates for quadrature formulas given by the ultraspherical polynomials. 
\end{remark}

\subsection{Gauss--Hermite quadrature}
The Gauss--Hermite quadrature for $I(f)$ is given by
\[ A_n(f)=\sum_{i=1}^n f(\xi_i)\rme^{{\xi_i}^2}w_i \approx \int_{-\infty}^{\infty} f(x)\rme^{x^2}\rme^{-x^2}\rd x = \int_{-\infty}^{\infty}f(x)\rd x=I(f),\]
where $\xi_{1},\ldots,\xi_{n}$ denote the Hermite nodes, i.e., the zeros of the physicist's Hermite polynomial of degree $n$, and the coefficients $w_1,\ldots,w_n$ are determined so that $A_n$ is exact for integrating polynomials of degree up to $2n-1$ with respect to $\mathrm{e}^{-x^2}\mathrm{d}x$. 

We use the same notation as in \eqref{eq:distinct-ordered} for Gauss--Hermite quadrature points:
\[ -\infty<\xi_{1}<\xi_{2}<\dots<\xi_n<\infty. \]
These nodes are known to satisfy the following properties, which we use to apply results from the previous section; see for example \cite[Section~6.31]{Szego1975}.
They are symmetric, i.e., $\xi_i=-\xi_{n-i+1}$ for all $i=1,\ldots,n$. For $n$ odd, $\xi_{(n-1)/2}=0$ and
\[ \frac{i\pi}{\sqrt{2n+1}}<\xi_{(n-1)/2+i}<\frac{4i+3}{\sqrt{2n+1}},\quad \text{for $i=1,\ldots,(n-1)/2$}\]
holds while 
for $n$ even,
\[ \frac{(i-1/2)\pi}{\sqrt{2n+1}}<\xi_{n/2+i}<\frac{4i+1}{\sqrt{2n+1}},\quad \text{for $i=1,\ldots,n/2$}\]
holds.
Moreover, for $n$ either odd or even, the minimum separation distance is bounded above and below by
\begin{align}\label{eq:xi_min_GH}
    \frac{\pi}{\sqrt{2n+1}}<\xi_{\min}\leq \frac{\sqrt{21/2}}{\sqrt{2n+1}}.
\end{align}

Applying Theorem~\ref{thm:lower_estimate1}, we obtain the following result.
\begin{corollary}\label{cor:gauss-hermite-general}
For any $d>0$
and $n\geq2$, the worst-case error of the Gauss--Hermite quadrature
in $H^{\infty}(\Dcal_{d},\omega)$ is bounded below by 
\[
c\,\frac{\exp\left(-2d\sqrt{2n+1}\right)}{(n+1/2)^{2}},
\]
for some constant $c>0$ that is independent of $n$, regardless of whether $\omega=\omega_{\SE}$ or $\omega=\omega_{\DE}$.
In particular,
for any constant $C>2\sqrt{2}d$, the worst-case error cannot decay
faster than $\exp(-C\sqrt{n})$. \end{corollary}

\begin{proof} From Theorem~\ref{thm:lower_estimate1} and \eqref{eq:xi_min_GH},
for any weight function $\omega$ the worst-case error of the Gauss--Hermite
quadrature is bounded below by 
\begin{align*}
 & \frac{(\tanh(1))^{4}}{30}\frac{\pi}{\sqrt{2n+1}}\min\biggl\{1,\biggl(\frac{\pi^{2}}{4d\sqrt{2n+1}}\biggr)^{4}\biggr\}\exp\left(-2d\sqrt{2n+1}\right)\\
 & \qquad\qquad\qquad\times\sum_{i=1}^{n-1}\omega(\max(|\xi_{i}|,|\xi_{i+1}|)).
\end{align*}
We will bound $\sum_{i=1}^{n-1}\omega(\max(|\xi_{i}|,|\xi_{i+1}|))$
from below. 

We first consider the case $H^{\infty}(\Dcal_{d},\omega_{\SE})$ with
$\beta=1$ and $\rho=2$, for which we can apply the bound we obtained
in \cite[Proof of Theorem~3.2]{KSG22}. Following \cite[Proof of Theorem~3.2]{KSG22},
it can be shown that there exists a constant $c'>0$, independent
of $n$, such that 
\[
\sum_{i=1}^{n-1}\omega_{\mathrm{SE}}(\max(|\xi_{i}|,|\xi_{i+1}|))=\sum_{i=1}^{n-1}\rme^{-\max(\xi_{i}^{2},\xi_{i+1}^{2})}\geq c'\sqrt{\pi(2n+1)}.
\]
Hence, we obtain the lower bound
\[
c'\frac{(\tanh(1))^{4}\pi^{3/2}}{30}\min\biggl\{1,\biggl(\frac{\pi^{2}}{4d\sqrt{2n+1}}\biggr)^{4}\biggr\}\exp\left(-2d\sqrt{2n+1}\right).
\]
This proves the statement for $\beta=1$ and $\rho=2$. 

The proof is analogous for other cases. For $n$ odd, the symmetry
of $\xi_{i}$'s, the fact $\xi_{(n-1)/2}=0$ and the monotonicity
of $\omega$ yield 
\begin{align*}
\sum_{i=1}^{n-1}\omega(\max(|\xi_{i}|,|\xi_{i+1}|)) & =2\sum_{i=1}^{(n-1)/2}\omega(\xi_{(n-1)/2+i})\geq2\int_{1}^{(n+1)/2}\omega\left(\frac{4x+3}{\sqrt{2n+1}}\right)\rd x\\
 & =\frac{\sqrt{2n+1}}{2}\int_{7/\sqrt{2n+1}}^{(2n+5)/\sqrt{2n+1}}\omega(y)\rd y,
\end{align*}
whereas, for $n$ even, it holds that 
\begin{align*}
\sum_{i=1}^{n-1}\omega(\max(|\xi_{i}|,|\xi_{i+1}|)) & =\omega(\xi_{n/2+1})+2\sum_{i=2}^{n/2}\omega(\xi_{n/2+i})\geq\int_{1}^{n/2+1}\omega\left(\frac{4x+1}{\sqrt{2n+1}}\right)\rd x\\
 & =\frac{\sqrt{2n+1}}{4}\int_{5/\sqrt{2n+1}}^{(2n+5)/\sqrt{2n+1}}\omega(y)\rd y.
\end{align*}
Then, an argument similar to the one made in \cite[Proof of Theorem~3.2]{KSG22}
reveals that these integrals are bounded below by a constant
$c'>0$ that depends only on $\alpha_{1},\beta,\rho$ for the SE case,
and only on $\alpha_{1},\beta_{1},\gamma$ for the DE case. 
This
concludes that the lower bound on the worst-case error given in Corollary~\ref{cor:gauss-hermite-general}
applies to these cases as well, just with different values for
$c$. 
\end{proof}

\begin{remark}
In \cref{subsec:upper_bounds}, we discussed the well-known Gauss–Hermite error decay rate of $O(-c\sqrt{n})$ for some constant $c > 0$ for analytic functions. 
Corollary~\ref{cor:gauss-hermite-general} states that even if the integrand decays as fast as $\omega_{\DE}$, the Gauss–Hermite quadrature error cannot decay faster than $O(-c\sqrt{n})$. Hence, Corollary~\ref{cor:gauss-hermite-general} consolidates the suboptimality of Gauss–Hermite quadrature for analytic functions pointed out by Trefethen \cite{Tre22}. 
This suboptimality of Gauss-type quadrature is a stark contrast to the numerical integration on bounded intervals, where for 
functions that are analytic in ellipses Gauss-type quadratures are essentially optimal; see \cite{Bakhvalov.NS_1967_OptimalSpeedIntegrating,Petras.K_1998_GaussianOptimalIntegration} for more details.
\end{remark}

\bibliographystyle{siamplain}
\bibliography{references}
\end{document}